\documentclass[reqno]{amsart}
\usepackage{amsfonts,amssymb,latexsym,amsmath,verbatim,url,color,mathrsfs,tikz,extarrows,enumerate,amscd,verbatim}
\usepackage[pdfpagelabels]{hyperref}
\usepackage[nameinlink]{cleveref}
\usetikzlibrary{matrix,arrows}
\numberwithin{equation}{section}

\def\Z{\mathbb{Z}}

\def\G{\mathbb{G}}
\def\H{\mathrm{H}}
\DeclareMathOperator{\Hom}{Hom}

\DeclareMathOperator{\End}{End}
\DeclareMathOperator{\Ann}{Ann}
\DeclareMathOperator{\Spec}{Spec}

\DeclareMathOperator{\id}{id}

\DeclareMathOperator{\Aut}{Aut}

\DeclareMathOperator{\Br}{Br}
\DeclareMathOperator{\Pic}{Pic}

\DeclareMathOperator{\LisEt}{Lis-Et}

\theoremstyle{plain}
\numberwithin{equation}{subsection}
\newtheorem{theorem}[subsection]{Theorem}

\newtheorem{proposition}[subsection]{Proposition}
\newtheorem{lemma}[subsection]{Lemma}

\newtheoremstyle{mystyle}
  {}
  {}
  {}
  {}
  {}
  {.}
  { }
  {\textbf{\thmname{#1}\thmnumber{ #2}}\thmnote{ (#3)}}
\theoremstyle{mystyle}
\newtheorem{definition}[subsection]{Definition}

\newtheorem{notation}[subsection]{Notation}
\newtheorem{remark}[subsection]{Remark}

\newtheorem{setup}[subsection]{Setup}

\newtheoremstyle{pgstyle}
  {}
  {}
  {}
  {}
  {}
  {.}
  { }
  {\textbf{\thmname{#1}\thmnumber{#2}}\thmnote{ (#3)}}
\theoremstyle{pgstyle}
\newtheorem{pg}[subsection]{}

\newcommand{\mf}[1]{\mathfrak{#1}}
\newcommand{\ms}[1]{\mathscr{#1}}
\newcommand{\mc}[1]{\mathcal{#1}}
\newcommand{\mb}[1]{\mathbf{#1}}
\newcommand{\mr}[1]{\mathrm{#1}}

\newcommand{\op}{\operatorname{op}}
\newcommand{\et}{\operatorname{\acute et}}

\begin{document}
\title{The cohomological Brauer group of a torsion $\G_{m}$-gerbe}
\author{Minseon Shin}
\email{shinms@math.berkeley.edu}
\urladdr{\url{http://math.berkeley.edu/~shinms}}
\address{Department of Mathematics \\ University of California, Berkeley \\ Berkeley, CA 94720-3840 USA}
\thanks{}
\keywords{Brauer group, gerbe, seminormal}
\subjclass[2010]{14F22, 14D23}
\begin{abstract} Let $S$ be a scheme and let $\pi : \mc{G} \to S$ be a $\G_{m,S}$-gerbe corresponding to a torsion class $[\mc{G}]$ in the cohomological Brauer group $\Br'(S)$ of $S$. We show that the cohomological Brauer group $\Br'(\mc{G})$ of $\mc{G}$ is isomorphic to the quotient of $\Br'(S)$ by the subgroup generated by the class $[\mc{G}]$. This is analogous to a theorem proved by Gabber for Brauer-Severi schemes. \end{abstract}
\date{\today}
\maketitle
\setcounter{tocdepth}{1}
\tableofcontents


\section{Introduction} \label{sec-01}

For an algebraic stack $\mc{X}$, let \[ \Br' \mc{X} := \H^{2}_{\et}(\mc{X} , \G_{m,\mc{X}})_{\mr{tors}} \] denote the \emph{cohomological Brauer group} of $\mc{X}$. Gabber in his thesis described the cohomological Brauer group of a Brauer-Severi scheme $X \to S$ as a quotient of the cohomological Brauer group of the base scheme $S$. More precisely, he proved the following: 

\begin{theorem} \label{20170930-29} \cite[Theorem 2]{GABBER-THESIS} Let $S$ be a scheme, let $\pi_{X} : X \to S$ be a Brauer-Severi scheme. Then the sequence \begin{align} \label{20170930-29-eqn-01} \H^{0}_{\et}(S,\Z) \to \Br' S \stackrel{\pi_{X}^{\ast}}{\to} \Br' X \to 0 \end{align} is exact, where the first map sends $1 \mapsto [X]$. \end{theorem}

In this paper we prove an analogue of the above theorem for torsion $\G_{m}$-gerbes.

\begin{theorem} \label{20170930-01} Let $S$ be a scheme, let $\pi_{\mc{G}} : \mc{G} \to S$ be a $\G_{m,S}$-gerbe corresponding to a torsion class $[\mc{G}] \in \Br' S$. Then the sequence \begin{align} \label{20170930-01-eqn-01} \H^{0}_{\et}(S,\Z) \to \Br' S \stackrel{\pi_{\mc{G}}^{\ast}}{\to} \Br' \mc{G} \to 0 \end{align} is exact, where the first map sends $1 \mapsto [\mc{G}]$. \end{theorem}

\begin{remark} \label{20170930-24} Let $\mc{A}$ be an Azumaya $\mc{O}_{S}$-algebra and let $\pi_{X} : X \to S$ and $\pi_{\mc{G}} : \mc{G} \to S$ be the associated Brauer-Severi scheme and $\G_{m}$-gerbe of trivializations, respectively. By \cite[\S8, 4]{QUILLEN-HAKT-I} there exists a finite locally free $\mc{O}_{X}$-module $J$ and an $\mc{O}_{X}$-algebra isomorphism $\pi_{X}^{\ast}\mc{A} \simeq \underline{\End}_{\mc{O}_{X}}(J)^{\op}$. There is an $\mc{O}_{X}$-algebra isomorphism $\underline{\End}_{\mc{O}_{X}}(J)^{\op} \simeq \underline{\End}_{\mc{O}_{X}}(J^{\vee})$ sending $\varphi \mapsto \varphi^{\vee}$. Since $\mc{G}$ is the gerbe of trivializations of $\mc{A}$, we have an $S$-morphism $f : X \to \mc{G}$; this induces a commutative triangle $\pi_{X}^{\ast} = f^{\ast}\pi_{\mc{G}}^{\ast}$ on the cohomological Brauer groups of $S,\mc{G},X$. Since $[\mc{G}] \in \ker \pi_{\mc{G}}^{\ast} \subseteq \ker \pi_{X}^{\ast}$ and $\ker \pi_{X}^{\ast}$ is generated by $[\mc{G}]$ by \Cref{20170930-29}, we have exactness of \labelcref{20170930-01-eqn-01} at $\Br' S$. The difficulty of \Cref{20170930-01} is in showing that $\pi_{\mc{G}}^{\ast}$ is surjective. \end{remark}

\begin{remark} \label{20170930-32} By \Cref{20170930-01} and \Cref{20170930-24}, the pullback map \[ f^{\ast} : \Br' \mc{G} \to \Br' X \] is an isomorphism, in other words the cohomological Brauer groups of a Brauer-Severi scheme and its associated $\G_{m}$-gerbe are isomorphic. \end{remark}

\begin{remark} \label{20170930-28} For an algebraic stack $\mc{X}$, let $\Br \mc{X}$ denote the \emph{(Azumaya) Brauer group} of $\mc{X}$, namely the abelian group of Morita equivalence classes of Azumaya $\mc{O}_{\mc{X}}$-algebras. There is a functorial injective homomorphism \[ \alpha_{\mc{X}} : \Br \mc{X} \to \Br' \mc{X} \] called the \emph{Brauer map} of $\mc{X}$, which sends an Azumaya $\mc{O}_{\mc{X}}$-algebra $\mc{A}$ to the associated $\G_{m,\mc{X}}$-gerbe of trivializations of $\mc{A}$. Then \Cref{20170930-01} provides a class of algebraic stacks $\mc{G}$ for which the Brauer map $\alpha_{\mc{G}}$ is surjective. Indeed, if $S$ is a scheme for which $\alpha_{S}$ is surjective and $\pi_{\mc{G}} : \mc{G} \to S$ is a torsion $\G_{m}$-gerbe, then $\alpha_{\mc{G}}$ is surjective by \Cref{20170930-01} and functoriality of the Brauer map. \end{remark}

\begin{pg} \label{20170930-31} We outline the proof of \Cref{20170930-01}. As in \cite{GABBER-THESIS}, the exact sequence \labelcref{20170930-01-eqn-01} comes from the Leray spectral sequence for the map $\pi_{\mc{G}}$ and sheaf $\G_{m,\mc{G}}$. One step in the proof of \Cref{20170930-01} is to show the vanishing of the higher pushforwards $\mb{R}^{2}\pi_{\mc{G},\ast}\G_{m,\mc{G}}$. The stalk of $\mb{R}^{2}\pi_{\mc{G},\ast}\G_{m,\mc{G}}$ at a geometric point $\overline{s}$ of $S$ is isomorphic to $\H^{2}_{\et}(\mr{B}\G_{m,A} , \G_{m,\mr{B}\G_{m,A}})$ where $A = \mc{O}_{S,\overline{s}}^{\mr{sh}}$ is the strict henselization of $S$ at $\overline{s}$. We compute $\H^{2}_{\et}(\mr{B}\G_{m,A} , \G_{m,\mr{B}\G_{m,A}})$ using the descent spectral sequence associated to the covering $\xi : \Spec A \to \mr{B}\G_{m,A}$, whose $q$th row is the \v{C}ech complex associated to the cosimplicial abelian group obtained by applying the functor $\H^{q}_{\et}(-,\G_{m})$ to the simplicial $A$-scheme $\{\G_{m,A}^{\times p}\}_{p \ge 0}$ obtained by taking fiber products of $\xi$. In \Cref{sec-03} and \Cref{sec-04}, we show that the $\mr{E}_{2}^{1,1}$ and $\mr{E}_{2}^{2,0}$ terms of this spectral sequence vanish, respectively. It is harder to show that $\mr{E}_{2}^{1,1} = 0$, which comes down to showing that \[ m^{\ast} - p_{1}^{\ast} - p_{2}^{\ast} : \Pic(A[t^{\pm}]) \to \Pic(A[t_{1}^{\pm},t_{2}^{\pm}]) \] is injective, where $m,p_{1},p_{2} : A[t^{\pm}] \to A[t_{1}^{\pm},t_{2}^{\pm}]$ are the $A$-algebra maps sending $t \mapsto t_{1}t_{2},t_{1},t_{2}$ respectively. If $A$ is a normal domain, then $\Pic(A) \simeq \Pic(A[t^{\pm}]) \simeq \Pic(A[t_{1}^{\pm},t_{2}^{\pm}])$ so the result is trivial. In case $A$ is not normal, we use the Units-Pic sequence associated to the Milnor square of the normalization $A \to \overline{A}$. We view \Cref{20170930-13} as the key lemma of this paper. \end{pg}

\begin{pg}[Acknowledgements] I thank my advisor Martin Olsson for suggesting this research topic, for generously sharing his knowledge and ideas, and for his patience. I am also grateful to Aise Johan de Jong, Ray Hoobler, Ariyan Javanpeykar, Max Lieblich, Siddharth Mathur, Lennart Meier, and Charles Weibel for helpful discussions and feedback. During this project, I received support from NSF grant DMS-1646385. \end{pg}

\section{Gerbes and the transgression map} \label{sec-02}

The purpose of this section is to prove \Cref{20171214-02}, a description of the higher pushforward $\mb{R}^{1}\pi_{\ast}\G_{m,\mc{G}}$ for a gerbe $\pi : \mc{G} \to \mc{S}$, and \Cref{20171214-03}, a description of the differential $\mr{d}_{2}^{0,1} : \mr{E}_{2}^{0,1} \to \mr{E}_{2}^{2,0}$ in the Leray spectral sequence associated to $\pi$ in terms of torsors and gerbes. This map $\mr{d}_{2}^{0,1}$ is called the \emph{transgression map} \cite[V, \S3.2]{GIRAUD-CN}. \par We first recall some background on gerbes. The standard reference is \cite{GIRAUD-CN}.

\begin{definition} \label{20171214-04} Let $\mc{S}$ be a site, and let $\pi : \mc{G} \to \mc{S}$ be a category fibered in groupoids. We view $\mc{G}$ as a site with the Grothendieck topology inherited from $\mc{S}$ \cite[III, 3.1]{SGA4}. For any object $U \in \mc{S}$, let $\mc{G}(U)$ denote the fiber category of $\mc{G}$ over $U$. \par The \emph{inertia stack} of $\pi : \mc{G} \to \mc{S}$ is the 2-fiber product $I_{\mc{G}/\mc{S}} := \mc{G} \times_{\Delta_{\mc{G}/\mc{S}},\mc{G} \times_{\mc{S}} \mc{G},\Delta_{\mc{G}/\mc{S}}} \mc{G}$ of the diagonal $\Delta_{\mc{G}/\mc{S}} : \mc{G} \to \mc{G} \times_{\mc{S}} \mc{G}$ with itself. The inertia stack $I_{\mc{G}/\mc{S}}$ is fibered in sets over $\mc{G}$ via either projection $I_{\mc{G}/\mc{S}} \to \mc{G}$, hence we may identify $I_{\mc{G}/\mc{S}}$ with the sheaf of groups on $\mc{G}$ associating $x \mapsto \Aut_{\mc{G}}(x)$. \par We say that $\pi$ is a \emph{gerbe} if the following conditions are satisfied: \begin{enumerate} \item[(i)] The fibered category $\mc{G}$ is a stack over $\mc{S}$. \item[(ii)] For any $U \in \mc{S}$, there exists a covering $\{U_{i} \to U\}_{i \in I}$ such that $\mc{G}(U_{i}) \ne \emptyset$ for all $i \in I$. \item[(iii)] For any $U \in \mc{S}$ and $x_{1},x_{2} \in \mc{G}(U)$, there exists a covering $\{U_{i} \to U\}_{i \in I}$ such that for all $i \in I$ there exists an isomorphism $x_{1}|_{U_{i}} \simeq x_{2}|_{U_{i}}$ in $\mc{G}(U_{i})$. \end{enumerate} Let $\mb{A}$ be an abelian sheaf on $\mc{S}$. We say that a gerbe $\pi$ is an \emph{$\mb{A}$-gerbe} if it is equipped with an isomorphism $\iota : \mb{A}_{\mc{G}} \to I_{\mc{G}/\mc{S}}$ of sheaves of groups on $\mc{G}$. \par If $\mc{S}$ is equipped with a sheaf of rings such that $(\mc{S},\mc{O}_{\mc{S}})$ is a locally ringed site, we set $\mc{O}_{\mc{G}} := \pi^{-1}\mc{O}_{\mc{S}}$; then the pair $(\mc{G},\mc{O}_{\mc{G}})$ is a locally ringed site. \qed \end{definition}

For the remainder of this section, we will assume the following setup:

\begin{setup} \label{20171214-05} Let $\mc{S}$ be a locally ringed site, let $\mb{A}$ be an abelian sheaf on $\mc{S}$, let $\pi : \mc{G} \to \mc{S}$ be an $\mb{A}$-gerbe. \end{setup}

\begin{definition}[Inertial action, eigensheaves, twisted sheaves] \label{20171214-06} \cite{LIEBLICH-THESIS}, \cite{LIEBLICH-TSAPIP} Let $\ms{F}$ be an $\mc{O}_{\mc{G}}$-module. For an object $x \in \mc{G}$ and $a \in \Gamma(x,\mb{A}_{\mc{G}})$, let $\iota(a)^{\ast} : \Gamma(x,\ms{F}) \to \Gamma(x,\ms{F})$ be the restriction map of the sheaf $\ms{F}$ via the automorphism $\iota(a) : x \to x$; such $x$ and $a$ defines an $\mc{O}_{\mc{G}/x}$-linear automorphism of $\ms{F}|_{\mc{G}/x}$ by $\{y \to x\} \mapsto \iota(a|_{y})^{\ast}$; thus we have a homomorphism $\mb{A}_{\mc{G}} \to \underline{\mr{Aut}}_{\mc{O}_{\mc{G}}}(\ms{F})$ of group sheaves on $\mc{G}$ corresponding to an $\mc{O}_{\mc{G}}$-linear $\mb{A}_{\mc{G}}$-action on $\ms{F}$, called the \emph{inertial action}. \par Let \[ \widehat{\mb{A}} := \Hom_{\mr{Ab}(\mc{S})}(\mb{A},\G_{m,\mc{S}}) \] denote the group of characters of $\mb{A}$. Given an $\mc{O}_{\mc{G}}$-module $\ms{F}$ and a character $\chi \in \widehat{\mb{A}}$, the $\chi$th \emph{eigensheaf} is the subsheaf $\ms{F}_{\chi} \subseteq \ms{F}$ defined as follows: for all objects $x \in \mc{G}$, a section $f \in \Gamma(x,\ms{F})$ is contained in $\Gamma(x,\ms{F}_{\chi})$ if for any morphism $y \to x$ and any $a \in \Gamma(y,\mb{A}_{\mc{G}})$ we have \begin{align} \label{20171214-06-eqn-01} (\iota(a)^{\ast})(f|_{y}) = \chi_{\mc{G}}(a) \cdot f|_{y} \end{align} in $\Gamma(y,\ms{F})$. The $\mc{O}_{\mc{G}}$-module $\ms{F}$ is called \emph{$\chi$-twisted} if the inclusion $\ms{F}_{\chi} \subseteq \ms{F}$ is an equality. If $\chi$ is the trivial character, the eigensheaf $\ms{F}_{\chi}$ is denoted $\ms{F}_{0}$ and $\chi$-twisted sheaves are called $0$-twisted. For any $\mc{O}_{\mc{S}}$-module $M$, the pullback $\pi^{\ast}M$ is $0$-twisted; in particular the structure sheaf $\mc{O}_{\mc{G}}$ is $0$-twisted. \qed \end{definition}

\begin{definition}[Category of $\chi$-twisted modules] \label{20171214-07} For a character $\chi \in \widehat{\mb{A}}$, let \[ \mr{Mod}(\mc{G},\chi) \] denote the full subcategory of $\mr{Mod}(\mc{G})$ consisting of $\chi$-twisted $\mc{O}_{\mc{G}}$-modules. \qed \end{definition}

\begin{remark} \label{20171214-11} Given two $\mc{O}_{\mc{G}}$-modules $\ms{F}$ and $\ms{G}$, any $\mc{O}_{\mc{G}}$-linear morphism $\varphi : \ms{F} \to \ms{G}$ restricts to an $\mc{O}_{\mc{G}}$-linear morphism $\varphi_{\chi} : \ms{F}_{\chi} \to \ms{G}_{\chi}$; the assignment $\ms{F} \mapsto \ms{F}_{\chi}$ defines a functor $\mr{Mod}(\mc{G}) \to \mr{Mod}(\mc{G},\chi)$ which is right adjoint (and a retraction) to the inclusion $\mr{Mod}(\mc{G},\chi) \to \mr{Mod}(\mc{G})$. \end{remark}

\begin{remark}[Modules on trivial gerbes] \label{20171214-08} We say that an $\mb{A}$-gerbe $\mc{G}$ is \emph{trivial} if there is an isomorphism $\mc{G} \simeq \mathrm{B}\mb{A}$. In this case we have the usual equivalence of categories between sheaves on $\mc{G}$ and sheaves on $\mc{S}$ equipped with an $\mb{A}$-action. For a sheaf $\ms{F} \in \mr{Sh}(\mr{B}\mb{A})$, the pushforward $\pi_{\ast}\ms{F}$ is identified with the subsheaf of $\ms{F}$ of sections invariant under the action of $\mb{A}$. For any sheaf $M \in \mr{Sh}(\mc{S})$, the inverse image $\pi^{-1}M \in \mr{Sh}(\mr{B}\mb{A})$ corresponds to the sheaf $M$ equipped with the trivial $\mb{A}$-action. If $s : \mc{S} \to \mr{B}\mb{A}$ is the section of $\pi$ corresponding to the trivial $\mb{A}$-torsor, then $s^{-1} : \mr{Sh}(\mr{B}\mb{A}) \to \mr{Sh}(\mc{S})$ is the functor forgetting the $\mb{A}$-action. \end{remark}

\begin{remark} \label{20171214-09} For any $\mc{O}_{\mc{G}}$-module $\ms{F}$, the counit map \[ \pi^{\ast}\pi_{\ast}\ms{F} \to \ms{F} \] is injective and its image coincides with $\ms{F}_{0}$. Indeed, this can be checked locally on $\mc{S}$, in which case we may assume $\mc{G}$ is the trivial gerbe and use \Cref{20171214-08}. \end{remark}

\begin{lemma} \label{20171214-10} Let $\mc{S}$ be a locally ringed site, let $\mb{A}$ be an abelian sheaf on $\mc{S}$, and let $\pi : \mc{G} \to \mc{S}$ be an $\mb{A}$-gerbe. The pullback functor \[ \pi^{\ast} : \mr{Mod}(\mc{S}) \to \mr{Mod}(\mc{G},0) \] is an equivalence of categories with quasi-inverse $\pi_{\ast}$. If $\mr{P}$ is a property of modules preserved by pullback via arbitrary morphisms of sites (e.g. quasi-coherent, flat, locally of finite type, locally of finite presentation, locally free), an $\mc{O}_{\mc{S}}$-module $M$ has $\mr{P}$ if and only if the $\mc{O}_{\mc{G}}$-module $\pi^{\ast}M$ has $\mr{P}$. \end{lemma} \begin{proof} For the first assertion, it suffices to show that for any $\mc{O}_{\mc{S}}$-module $M$ the unit map \[ M \to \pi_{\ast}\pi^{\ast}M \] is an isomorphism, and that for any $0$-twisted $\mc{O}_{\mc{G}}$-module $\ms{F}$ the counit map \[ \pi^{\ast}\pi_{\ast}\ms{F} \to \ms{F} \] is an isomorphism. Both of these claims are local on $\mc{S}$, hence we may assume that $\mc{G}$ is trivial, in which case the claims follow from \Cref{20171214-08} and \Cref{20171214-09}. The second assertion is also local on $\mc{S}$, hence we may assume that $\mc{G}$ is trivial, in which case there is a section $s : \mc{S} \to \mc{G}$ of $\pi$. For any $0$-twisted $\mc{O}_{\mc{G}}$-module $\ms{F}$, we have $\pi_{\ast}\ms{F} \simeq s^{\ast}\ms{F}$ by the discussion in \Cref{20171214-08}. \end{proof}

In order to describe the higher pushforwards $\mb{R}^{1}\pi_{\ast}\G_{m}$, we will use the following result on the Picard group of $\mb{A}$-gerbes.

\begin{remark}[Picard group of $\mb{A}$-gerbes] \label{20171214-01} Assume the setup of \Cref{20171214-10}. By \cite[5.3.4]{BROCHARD-FDPDCA}, for any invertible $\mc{O}_{\mc{G}}$-module $\mc{L}$, there exists a unique character \[ \chi_{\mc{L}} \in \widehat{\mb{A}} \] such that the diagram \begin{equation} \label{20171214-01-eqn-02} \begin{tikzpicture}[>=angle 90, baseline=(current bounding box.center)] 
\matrix[matrix of math nodes,row sep=2em, column sep=2em, text height=1.5ex, text depth=0.25ex] { 
|[name=11]| \mb{A}_{\mc{G}} \times \mc{L} & |[name=12]| \mc{L} \\ 
|[name=21]| \G_{m,\mc{G}} \times \mc{L} & |[name=22]| \mc{L} \\
}; 
\draw[->,font=\scriptsize] (11) edge (12) (21) edge (22) (11) edge node[left=-1pt] {$\pi^{\ast}\chi_{\mc{L}} \times \id_{\mc{L}}$} (21) (12) edge node[right=-1pt] {$\id_{\mc{L}}$} (22); \end{tikzpicture} \end{equation} commutes, where the top row is the inertial action and the bottom row is the restriction of the $\mc{O}_{\mc{G}}$-module structure on $\mc{L}$. The condition that \labelcref{20171214-01-eqn-02} commutes is equivalent to the condition \labelcref{20171214-06-eqn-01}, in other words $\mc{L}$ is a $\chi_{\mc{L}}$-twisted sheaf. For two invertible $\mc{O}_{\mc{G}}$-modules $\mc{L}_{1},\mc{L}_{2}$ we have $\chi_{\mc{L}_{1} \otimes \mc{L}_{2}} = \chi_{\mc{L}_{1}} \cdot \chi_{\mc{L}_{2}}$ by \cite[5.3.6 (2)]{BROCHARD-FDPDCA}, hence the assignment $\mc{L} \mapsto \chi_{\mc{L}}$ defines a group homomorphism \[ \beta_{\mc{G}} : \Pic(\mc{G}) \to \widehat{\mb{A}} \] of abelian groups. By \cite[5.3.6 (3)]{BROCHARD-FDPDCA}, we have that $\chi_{\mc{L}} = 0$ if and only if $\mc{L}$ is of the form $\pi^{\ast}M$ for an invertible $\mc{O}_{\mc{S}}$-module $M$; in other words there is an exact sequence \begin{align} \label{20171214-01-eqn-01} 0 \to \Pic(\mc{S}) \stackrel{\pi^{\ast}}{\to} \Pic(\mc{G}) \stackrel{\beta_{\mc{G}}}{\to} \widehat{\mb{A}} \end{align} where injectivity of $\pi^{\ast}$ follows from \Cref{20171214-10}. The sequence \labelcref{20171214-01-eqn-01} is functorial on $\mc{S}$ in the following sense: if $p : \mc{T} \to \mc{S}$ is a morphism of locally ringed sites and $\pi_{\mc{T}} : \mc{G}_{\mc{T}} \to \mc{T}$ is the $\mb{A}_{\mc{T}}$-gerbe obtained by pullback, then the diagram \begin{center}\begin{tikzpicture}[>=angle 90] 
\matrix[matrix of math nodes,row sep=2em, column sep=2em, text height=1.8ex, text depth=0.25ex] { 
|[name=11]| 0 & |[name=12]| \Pic(\mc{S}) & |[name=13]| \Pic(\mc{G}) & |[name=14]| \widehat{\mb{A}} \\ 
|[name=21]| 0 & |[name=22]| \Pic(\mc{T}) & |[name=23]| \Pic(\mc{G}_{\mc{T}}) & |[name=24]| \widehat{\mb{A}}_{\mc{T}} \\ 
}; 
\draw[->,font=\scriptsize] (11) edge (12) (12) edge node[above=-1pt] {$\pi^{\ast}$} (13) (13) edge node[above=-1pt] {$\beta_{\mc{G}}$} (14) (21) edge (22) (22) edge node[below=-1pt] {$\pi_{\mc{T}}^{\ast}$} (23) (23) edge node[below=-1pt] {$\beta_{\mc{G}_{\mc{T}}}$} (24) (12) edge node[left=-1pt] {$p^{\ast}$} (22) (13) edge node[left=-1pt] {$p^{\ast}$} (23) (14) edge node[left=-1pt] {$p^{\ast}$} (24); \end{tikzpicture} \end{center} commutes. \par In case $\mc{G} := \mr{B}\mb{A}$, by \cite[5.3.7]{BROCHARD-FDPDCA} the map $\beta_{\mc{G}}$ is surjective and the sequence \labelcref{20171214-01-eqn-01} is split; the map $\beta_{\mc{G}}$ admits a natural section $\widehat{\mb{A}} \to \Pic(\mc{G})$ taking a character $\chi$ to the trivial $\mc{O}_{\mc{S}}$-module equipped with the $\mb{A}$-action corresponding to $\chi$ via the isomorphism $\G_{m,\mc{S}} \simeq \underline{\Aut}_{\mc{O}_{\mc{S}}}(\mc{O}_{\mc{S}})$. \end{remark}

\begin{lemma} \label{20171214-02} Let $\mc{S}$ be a locally ringed site, let $\mb{A}$ be an abelian sheaf on $\mc{S}$, let $\pi : \mc{G} \to \mc{S}$ be an $\mb{A}$-gerbe. There is a natural isomorphism \begin{align} \label{20171214-02-eqn-01} \mb{R}^{1}\pi_{\ast}\G_{m,\mc{G}} \simeq \underline{\Hom}_{\mr{Ab}(\mc{S})}(\mb{A},\G_{m,\mc{S}}) \end{align} of abelian sheaves on $\mc{S}$. \end{lemma} \begin{proof} Let $U \in \mc{S}$ be an object. Taking $\mc{T} := \mc{S}/U$ and $p : \mc{S}/U \to \mc{S}$ the inclusion of categories, we obtain an exact sequence \begin{align} \label{20171214-02-eqn-02} 0 \to \Pic(\mc{S}/U) \stackrel{\pi_{\mc{S}/U}^{\ast}}{\to} \Pic(\mc{G}_{\mc{S}/U}) \stackrel{\beta_{\mc{G}_{\mc{S}/U}}}{\to} \widehat{\mb{A}}_{\mc{S}/U} \end{align} of abelian groups. Letting $U$ range over the objects of $\mc{S}$, we obtain an exact sequence of abelian presheaves whose value on $U$ is \labelcref{20171214-02-eqn-02}, and sheafifying this sequence gives the desired isomorphism. \end{proof}

We specialize to the case $\mb{A} = \G_{m,\mc{S}}$.

\begin{proposition} \label{20171214-03} Let $\mc{S}$ be a locally ringed site and let $\pi : \mc{G} \to \mc{S}$ be a $\G_{m,\mc{S}}$-gerbe. Let \begin{align} \label{20171214-03-eqn-01} \mr{d}_{2}^{0,1} : \H^{0}(\mc{S} , \mb{R}^{1}\pi_{\ast}\G_{m,\mc{G}}) \to \H^{2}(\mc{S} , \mb{R}^{0}\pi_{\ast}\G_{m,\mc{G}}) \end{align} be the differential in the Leray spectral sequence associated to the map $\pi$ and sheaf $\G_{m,\mc{G}}$. Under the identification \labelcref{20171214-02-eqn-01}, the differential $\mr{d}_{2}^{0,1}$ sends the identity $\id_{\G_{m,\mc{S}}} \in \Hom_{\mr{Ab}(\mc{S})}(\G_{m,\mc{S}},\G_{m,\mc{S}})$ to the class $[\mc{G}] \in \H^{2}(\mc{S} , \G_{m,\mc{S}})$. \end{proposition} \begin{proof} Let $c \in \H^{0}(\mc{S} , \mb{R}^{1}\pi_{\ast}\G_{m,\mc{G}})$ be the class corresponding to the identity section $\chi := \id_{\G_{m,\mc{S}}} \in \Hom_{\mr{Ab}(\mc{S})}(\G_{m,\mc{S}},\G_{m,\mc{S}})$ via the isomorphism \labelcref{20171214-02-eqn-01}. Let $D(c) \to \mc{S}$ denote the category fibered in groupoids whose fiber category $(D(c))(U)$ for an object $U \in \mc{S}$ consists of the invertible $\mc{O}_{\mc{G}_{\mc{S}/U}}$-modules whose image under the map \[ \H^{1}(\mc{G}_{\mc{S}/U} , \G_{m,\mc{S}}) \to \H^{0}(U , \mb{R}^{1}\pi_{\ast}\G_{m,\mc{G}}) \] is equal to the image of $c$ under the restriction map \[ \H^{0}(\mc{S} , \mb{R}^{1}\pi_{\ast}\G_{m,\mc{G}}) \to \H^{0}(U , \mb{R}^{1}\pi_{\ast}\G_{m,\mc{G}}) \] of the sheaf $\mb{R}^{1}\pi_{\ast}\G_{m,\mc{G}}$. By \cite[V, 3.2.1]{GIRAUD-CN}, the category $D(c)$ is a $\G_{m,\mc{S}}$-gerbe, and the assignment $c \mapsto [D(c)]$ coincides with the differential \labelcref{20171214-03-eqn-01}. By the above description of $D(c)$ and by the definition of the isomorphism \labelcref{20171214-02-eqn-01} as the one obtained by sheafifying the maps $\beta_{\mc{G}_{\mc{S}/U}}$ in \labelcref{20171214-02-eqn-02}, we have that an invertible $\mc{O}_{\mc{G}_{\mc{S}/U}}$-module $\mc{L}$ is contained in $(D(c))(U)$ if and only if it is $\chi|_{\mc{S}/U}$-twisted. By \cite[Proposition 2.1.2.5]{LIEBLICH-THESIS}, we have that there is an isomorphism $\mc{G} \simeq D(c)$ of $\G_{m,\mc{S}}$-gerbes. \end{proof}

\section{Picard groups of (Laurent) polynomial rings} \label{sec-03}

In this section we prove \Cref{20170930-04}. For us, the main difficulty is that there are rings $A$ for which the pullback map $\Pic(A) \to \Pic(A[t])$ is not an isomorphism. The ring $A$ is called \emph{seminormal} \cite[p. 210]{SWAN-ONSEMINORMALITY}, \cite[p. 29]{WEIBEL-KBOOK} if for every $b,c \in A$ satisfying $b^{3} = c^{2}$ there exists $a \in A$ such that $a^{2} = b$ and $a^{3} = c$. Seminormal rings are automatically reduced \cite[VIII, \S7]{LAM-SERRES-PROBLEM}. By Traverso's theorem \cite[Theorem 3.6]{TRAVERSO-SAPG}, \cite[Theorem 3.11]{WEIBEL-KBOOK}, the map $\Pic(A) \to \Pic(A[t])$ is an isomorphism if and only if the reduction $A_{\mr{red}}$ is a seminormal ring. Taking the strict henselization of the cuspidal cubic $k[x,y]/(y^{2} = x^{3})$ at the cusp gives an example of a reduced strictly henselian local ring $A$ which is not seminormal; by \Cref{20170930-26}, in this case we also have $\Pic(A[t,t^{-1}]) \ne 0$.

Throughout this section and \Cref{sec-04}, we will use \Cref{20170930-33} and \Cref{20170930-14}.

\begin{notation}[$\Delta,\mr{C}^{\bullet}\mb{G},\mathsf{h}^{n}(\mr{C}^{\bullet}\mb{G})$] \label{20170930-33} Let $\Delta$ be the category with objects $[n] := \{0,\dotsc,n\}$ for each nonnegative integer $n \ge 0$ and whose morphisms $\varphi : [m] \to [n]$ correspond to nondecreasing maps $\varphi : \{0,\dotsc,m\} \to \{0,\dotsc,n\}$. For $n \ge 0$ and $0 \le i \le n+1$, we denote $\delta_{i}^{n} : [n] \to [n+1]$ the injective nondecreasing map whose image does not contain $i$. For $n \ge 0$ and $0 \le i \le n$, we denote $\sigma_{i}^{n} : [n+1] \to [n]$ the surjective nondecreasing map satisfying $(\sigma_{i}^{n})^{-1}(i) = \{i,i+1\}$. A cosimplicial set (resp. abelian group, resp. ring) is a covariant functor from $\Delta$ to $(\mr{Set})$ (resp. $(\mr{Ab})$, resp. $(\mr{Ring})$). \par If $\mb{G}$ is a cosimplicial abelian group, we denote by \[ \mr{C}^{\bullet}\mb{G} \] the cochain complex where $\mr{C}^{n}\mb{G} := \mb{G}([n])$ for $n \ge 0$ and where the $n$th differential $\mb{d}_{\mb{G}}^{n} : \mr{C}^{n}\mb{G} \to \mr{C}^{n+1}\mb{G}$ is the alternating sum $\sum_{i=0}^{n+1} \mb{G}(\delta_{i}^{n})$. We denote by \[ \mathsf{h}^{n}(\mr{C}^{\bullet}\mb{G}) \] the cohomology of $\mr{C}^{\bullet}\mb{G}$ at $\mr{C}^{n}\mb{G}$. \end{notation}

\begin{notation}[$\mb{L}_{A},\mb{P}_{A}$] \label{20170930-14} Let $A$ be a ring, let $\pi : \mr{B}\G_{m,A} \to \Spec A$ be the trivial $\G_{m,A}$-gerbe, let $\xi : \Spec A \to \mr{B}\G_{m,A}$ be the section of $\pi$ corresponding to the trivial $\G_{m,A}$-torsor. Taking 2-fiber products of $\xi$, we obtain a cosimplicial $A$-algebra \[ \mb{L}_{A} : \Delta \to (A\text{-alg}) \] where \[ \mb{L}_{A}([p]) := A[t_{1}^{\pm},\dotsc,t_{p}^{\pm}] \] is the Laurent polynomial ring in $p$ indeterminates over $A$ (where by convention $\mb{L}_{A}([0]) := A$). For $p \ge 0$ and $0 \le i \le p+1$, the $i$th degeneracy map $\mb{L}_{A}(\delta_{i}^{p}) : \mb{L}_{A}([p]) \to \mb{L}_{A}([p+1])$ is the $A$-algebra map sending $(t_{1},\dotsc,t_{p}) \mapsto (t_{1},\dotsc,t_{i}t_{i+1},\dotsc,t_{p+1})$ where by abuse of notation we write ``$t_{0}$'' and ``$t_{p+2}$'' to mean ``$1$'' (in the cases $i=0$ and $i=p+1$ respectively). For $p \ge 0$ and $0 \le i \le p$, the $i$th face map $\mb{L}_{A}(\sigma_{i}^{p}) : \mb{L}_{A}([p+1]) \to \mb{L}_{A}([p])$ is the $A$-algebra map sending $(t_{1},\dotsc,t_{p+1}) \mapsto (t_{1},\dotsc,t_{i},1,t_{i+1},\dotsc,t_{p})$. \par We also have the cosimplicial $A$-algebra \[ \mb{P}_{A} : \Delta \to (A\text{-alg}) \] where \[ \mb{P}_{A}([p]) := A[t_{1},\dotsc,t_{p}] \] is the polynomial ring in $p$ indeterminates over $A$, viewed as the subalgebra of $\mb{L}_{A}([p])$, and for which the $A$-algebra map $\mb{P}_{A}(\varphi) : \mb{P}_{A}([m]) \to \mb{P}_{A}([n])$ is obtained by restricting $\mb{L}_{A}(\varphi) : \mb{L}_{A}([m]) \to \mb{L}_{A}([n])$. \par We make the formulas $\mb{P}_{A}(\delta_{i}^{p})$ and $\mb{P}_{A}(\sigma_{i}^{p})$ explicit for $p=0,1$. For $0 \le i \le 1$, the $A$-algebra map $\mb{P}_{A}(\delta_{i}^{0}) : A \to A[t_{1}]$ is the unique one. For $0 \le i \le 2$, the $A$-algebra map $\mb{P}_{A}(\delta_{i}^{1}) : A[t_{1}] \to A[t_{1},t_{2}]$ sends $t_{1}$ to $t_{1},t_{1}t_{2},t_{2}$ respectively. For $0 \le i \le 0$, the $A$-algebra map $\mb{P}_{A}(\sigma_{i}^{0}) : A[t_{1}] \to A$ sends $t_{1}$ to $1$. For $0 \le i \le 1$, the $A$-algebra map $\mb{P}_{A}(\sigma_{i}^{1}) : A[t_{1},t_{2}] \to A[t_{1}]$ sends $(t_{1},t_{2})$ to $(1,t_{1}),(t_{1},1)$ respectively. \end{notation}

\begin{notation}[$\mr{N}_{x}\mr{F},\mr{N}_{x_{1},x_{2}}\mr{F}$] \label{20170930-34} Given a functor \begin{align} \label{20170930-14-eqn-01} \mr{F} : (\mr{Ring}) \to (\mr{Ab}) \end{align} we define new functors \[ \mr{N}_{x}\mr{F} , \mr{N}_{x_{1},x_{2}}\mr{F} : (\mr{Ring}) \to (\mr{Ab}) \] by \begin{align*} \mr{N}_{x}\mr{F}(A) &:= \ker(\mr{F}(x=1) : \mr{F}(A[x]) \to \mr{F}(A)) \\ \mr{N}_{x_{1},x_{2}}\mr{F}(A) &:= \ker((\mr{F}(x_{2}=1),\mr{F}(x_{1}=1)) : \mr{F}(A[x_{1},x_{2}]) \to \mr{F}(A[x_{1}]) \oplus \mr{F}(A[x_{2}])) \end{align*} for any ring $A$, where $x,x_{1},x_{2}$ are indeterminates. The notation ``$\mr{N}_{x}\mr{F}$'' was defined by Weibel in \cite[\S1]{WEIBEL-PIACF}. \end{notation}

The operation ``$\mr{N}_{x}$'' can be iterated, for example if $x_{1},x_{2}$ are indeterminates, then $\mr{N}_{x_{1}}(\mr{N}_{x_{2}}\mr{F})$ is a functor $(\mr{Ring}) \to (\mr{Ab})$.

\begin{lemma} \label{20170930-35} In \Cref{20170930-34}, we have \[ \mr{N}_{x_{1},x_{2}}\mr{F}(A) = \mr{N}_{x_{1}}(\mr{N}_{x_{2}}\mr{F})(A) = \mr{N}_{x_{2}}(\mr{N}_{x_{1}}\mr{F})(A) \] for any ring $A$. \end{lemma} \begin{proof} The claim follows from considering the commutative diagram \begin{center}\begin{tikzpicture}[>=angle 90] 
\matrix[matrix of math nodes,row sep=2.5em, column sep=2em, text height=1.5ex, text depth=0.25ex] { 
|[name=00]| & |[name=01]| 0 & |[name=02]| 0 & |[name=03]| 0 & |[name=04]| \\
|[name=10]| 0 & |[name=11]| \mr{N}_{x_{1},x_{2}}\mr{F}(A) & |[name=12]| \mr{N}_{x_{1}}\mr{F}(A[x_{2}]) & |[name=13]| \mr{N}_{x_{1}}\mr{F}(A) & |[name=14]| 0 \\ 
|[name=20]| 0 & |[name=21]| \mr{N}_{x_{2}}\mr{F}(A[x_{1}]) & |[name=22]| \mr{F}(A[x_{1},x_{2}]) & |[name=23]| \mr{F}(A[x_{1}]) & |[name=24]| 0 \\ 
|[name=30]| 0 & |[name=31]| \mr{N}_{x_{2}}\mr{F}(A) & |[name=32]| \mr{F}(A[x_{2}]) & |[name=33]| \mr{F}(A) & |[name=34]| 0 \\ 
|[name=40]| & |[name=41]| 0 & |[name=42]| 0 & |[name=43]| 0 & |[name=44]| \\ 
}; 
\draw[->,font=\scriptsize]
(01) edge (11) (02) edge (12) (03) edge (13) (31) edge (41) (32) edge (42) (33) edge (43) (10) edge (11) (13) edge (14) (20) edge (21) (23) edge (24) (30) edge (31) (33) edge (34)
(11) edge (12) (12) edge node[above=-1pt] {$x_{2}=1$} (13)
(21) edge (22) (22) edge node[above=-1pt] {$x_{2}=1$} (23)
(31) edge (32) (32) edge node[above=-1pt] {$x_{2}=1$} (33)
(11) edge (21) (21) edge node[left=-1pt] {$x_{1}=1$} (31)
(12) edge (22) (22) edge node[left=-1pt] {$x_{1}=1$} (32)
(13) edge (23) (23) edge node[left=-1pt] {$x_{1}=1$} (33); \end{tikzpicture} \end{center} where each row and column is (split) exact. \end{proof}

\begin{lemma} \label{20170930-23} Assume \Cref{20170930-33}, \Cref{20170930-14}, and \Cref{20170930-34}. We have \[ \mb{d}_{\mr{F}\mb{P}_{A}}^{1}(\mr{N}_{t_{1}}\mr{F}(A)) \subset \mr{N}_{t_{1},t_{2}}\mr{F}(A) \] for any ring $A$. \end{lemma} \begin{proof} For $0 \le i \le 2$, the composition $\mb{P}_{A}(\sigma_{0}^{1}) \mb{P}_{A}(\delta_{i}^{1})$ correspond to the $A$-algebra maps $A[t_{1}] \to A[t_{1}]$ sending $t_{1} \mapsto 1,t_{1},t_{1}$, respectively; thus $\mr{F}(\mb{P}_{A}(\sigma_{0}^{1}))(\mb{d}_{\mr{F}\mb{P}_{A}}^{1}(\mr{N}_{t_{1}}\mr{F}(A))) = 0$. By a similar argument, we have $\mr{F}(\mb{P}_{A}(\sigma_{1}^{1}))(\mb{d}_{\mr{F}\mb{P}_{A}}^{1}(\mr{N}_{t_{1}}\mr{F}(A))) = 0$. \end{proof}

\begin{lemma} \label{20170930-13} Assume \Cref{20170930-33}, \Cref{20170930-14}, and \Cref{20170930-34}. We have \[ \mathsf{h}^{1}(\mr{C}^{\bullet}(\mr{Pic}\mb{P}_{A})) = 0 \] for any ring $A$. \end{lemma} \begin{proof} Since $\mb{P}_{A}(\delta_{0}^{0}) = \mb{P}_{A}(\delta_{1}^{0})$, the differential $\mb{d}_{\mr{Pic}\mb{P}_{A}}^{0} : \Pic(A) \to \Pic(A[t_{1}])$ is the $0$ map. Hence it suffices to show that \[ \mb{d}_{\mr{Pic}\mb{P}_{A}}^{1} : \Pic(A[t_{1}]) \to \Pic(A[t_{1},t_{2}]) \] is injective. \par We have that $A$ is the filtered colimit of subrings of $A$ which are finite type $\Z$-algebras, hence by e.g. \cite[0B8W]{STACKS-PROJECT} we may reduce to the case when $A$ is a finite type $\Z$-algebra. In particular $A$ has finite Krull dimension. We proceed by induction on $\dim A$. Since the Picard group of a ring is invariant under nilpotent thickenings, we may assume that $A$ is reduced. If $\dim A = 0$, then $A$ is a finite product of fields, hence $\ker \mb{d}_{\mr{Pic}\mb{P}_{A}}^{1} = 0$ (since in fact $\Pic(A[t]) = 0$ in this case). \par Suppose $\dim A > 0$ and let \[ \alpha \in \mr{Pic}(A[t_{1}]) \] be a class such that $\mb{d}_{\mr{Pic}\mb{P}_{A}}^{1}(\alpha) = 0$. We have a direct sum decomposition $\Pic(A) \oplus \mr{N}_{t_{1}}\mr{Pic}(A) \simeq \Pic(A[t_{1}])$, and $\mb{P}_{A}(\delta_{0}^{1}),\mb{P}_{A}(\delta_{1}^{1}),\mb{P}_{A}(\delta_{2}^{1})$ are $A$-algebra maps, so in fact $\alpha \in \mr{N}_{t_{1}}\mr{Pic}(A)$. Let $Q(A)$ denote the total ring of fractions of $A$, and let \[ A^{\mr{sn}} \subset Q(A) \] denote the seminormalization \cite[Lemma 2.2]{SWAN-ONSEMINORMALITY} of $A$ in $Q(A)$. Write \[ \textstyle A^{\mr{sn}} = \varinjlim_{\lambda \in \Lambda} A_{\lambda} \] where each $A \subset A_{\lambda} \subset A^{\mr{sn}}$ is a finitely generated subextension of $A^{\mr{sn}}$; then $A \subset A_{\lambda}$ is a finite extension of rings since it is an integral extension. Thus \[ \textstyle \mr{N}_{t_{1}}\mr{Pic}(A^{\mr{sn}}) \simeq \varinjlim_{\lambda \in \Lambda} \mr{N}_{t_{1}}\mr{Pic}(A_{\lambda}) \] by e.g. \cite[0B8W]{STACKS-PROJECT}. By \cite[Corollary 3.4]{SWAN-ONSEMINORMALITY}, we have that $A^{\mr{sn}}$ is seminormal, thus $\mr{N}_{t_{1}}\mr{Pic}(A^{\mr{sn}}) = 0$ by Traverso's theorem \cite[Theorem 3.11]{WEIBEL-KBOOK}. Hence there exists some $\lambda \in \Lambda$ for which $\alpha$ lies in the kernel of $\mr{N}_{t_{1}}\mr{Pic}(A) \to \mr{N}_{t_{1}}\mr{Pic}(A_{\lambda})$.\footnote{Here, instead of using the limit argument, we may also use that the extension $A \subset A^{\rm{sn}}$ is finite since $A$ is a Nagata ring (it is a finite type $\Z$-algebra) and thus has finite normalization, hence has finite seminormalization.} Let \[ I := \{x \in A \;:\; xA_{\lambda} \subset A\} = \Ann_{A}(A_{\lambda}/A) \] be the conductor ideal of $A \subset A_{\lambda}$; it is the largest ideal of $A_{\lambda}$ contained in $A$ so in particular it is also an ideal of $A$. We denote \[ \mr{U}(A) := A^{\times} \] the group of units of $A$. By Milnor's theorem \cite[IX, (5.3)]{BASS-AKT}, the Milnor square \begin{equation} \label{20170930-13-eqn-04} \begin{tikzpicture}[>=angle 90, baseline=(current bounding box.center)] 
\matrix[matrix of math nodes,row sep=3em, column sep=2em, text height=1.5ex, text depth=0.25ex] { 
|[name=11]| A & |[name=12]| A_{\lambda} \\ 
|[name=21]| A/I & |[name=22]| A_{\lambda}/I \\
}; 
\draw[->,font=\scriptsize]
(11) edge[right hook->] (12) (21) edge[right hook->] (22) (11) edge[->>] (21) (12) edge[->>] (22); \end{tikzpicture} \end{equation} gives an exact sequence \begin{align} \label{20170930-13-eqn-03} \begin{aligned} 1 &\to \mr{U}(A) \stackrel{\Delta}{\to} \mr{U}(A/I) \oplus \mr{U}(A_{\lambda}) \stackrel{\pm}{\to} \mr{U}(A_{\lambda}/I) \\ &\stackrel{\partial}{\to} \Pic(A) \stackrel{\Delta}{\to} \Pic(A/I) \oplus \Pic(A_{\lambda}) \stackrel{\pm}{\to} \Pic(A_{\lambda}/I) \end{aligned} \end{align} of abelian groups, called the Units-Pic sequence \cite[I, Theorem 3.10]{WEIBEL-KBOOK}; here we denote by $\Delta$ the diagonal map and by $\pm$ the difference map. The boundary map $\partial$ of \labelcref{20170930-13-eqn-03} is functorial for morphisms between Milnor squares. Hence applying $\mr{N}_{t_{1}}$ and $\mr{N}_{t_{1},t_{2}}$ to \labelcref{20170930-13-eqn-03} gives a commutative diagram \begin{equation} \label{20170930-13-eqn-01} \begin{tikzpicture}[>=angle 90, baseline=(current bounding box.center)] 
\matrix[matrix of math nodes,row sep=3em, column sep=5em, text height=1.6ex, text depth=0.5ex] { 
|[name=11]| \mr{N}_{t_{1}}\mr{U}(A/I) \oplus \mr{N}_{t_{1}}\mr{U}(A_{\lambda}) & |[name=12]| \mr{N}_{t_{1},t_{2}}\mr{U}(A/I) \oplus \mr{N}_{t_{1},t_{2}}\mr{U}(A_{\lambda}) \\ 
|[name=21]| \mr{N}_{t_{1}}\mr{U}(A_{\lambda}/I) & |[name=22]| \mr{N}_{t_{1},t_{2}}\mr{U}(A_{\lambda}/I) \\
|[name=31]| \mr{N}_{t_{1}}\mr{Pic}(A) & |[name=32]| \mr{N}_{t_{1},t_{2}}\mr{Pic}(A) \\
|[name=41]| \mr{N}_{t_{1}}\mr{Pic}(A/I) \oplus \mr{N}_{t_{1}}\mr{Pic}(A_{\lambda}) & |[name=42]| \mr{N}_{t_{1},t_{2}}\mr{Pic}(A/I) \oplus \mr{N}_{t_{1},t_{2}}\mr{Pic}(A_{\lambda}) \\
}; 
\draw[->,font=\scriptsize] (11) edge node[above=-1pt] {$\mb{d}_{\mr{U}\mb{P}_{A/I}}^{1} \oplus \mb{d}_{\mr{U}\mb{P}_{A_{\lambda}}}^{1}$} (12) (21) edge node[above=-1pt] {$\mb{d}_{\mr{U}\mb{P}_{A_{\lambda}/I}}^{1}$} (22) (31) edge node[above=-1pt] {$\mb{d}_{\mr{Pic}\mb{P}_{A}}^{1}$} (32) (41) edge node[above=4pt] {$\mb{d}_{\mr{Pic}\mb{P}_{A/I}}^{1} \oplus \mb{d}_{\mr{Pic}\mb{P}_{A_{\lambda}}}^{1}$} (42) (11) edge node[left=-1pt] {$\pm_{t_{1}}$} (21) (21) edge node[left=-1pt] {$\partial_{t_{1}}$} (31) (31) edge node[left=-1pt] {$\Delta_{t_{1}}$} (41) (12) edge node[right=-1pt] {$\pm_{t_{1},t_{2}}$} (22) (22) edge node[right=-1pt] {$\partial_{t_{1},t_{2}}$} (32) (32) edge node[right=-1pt] {$\Delta_{t_{1},t_{2}}$} (42); \end{tikzpicture} \end{equation} where we denote by $\pm_{t_{1}},\partial_{t_{1}},\Delta_{t_{1}}$ (resp. $\pm_{t_{1},t_{2}},\partial_{t_{1},t_{2}},\Delta_{t_{1},t_{2}}$) the corresponding maps in the Units-Pic sequences associated to the Milnor squares obtained by tensoring \labelcref{20170930-13-eqn-04} with $- \otimes_{\Z} \Z[t_{1}]$ (resp. $- \otimes_{\Z} \Z[t_{1},t_{2}]$), and we denote by $\mb{d}_{\mr{F}\mb{P}_{A}}^{1}$ the restriction to $\mr{N}_{t_{1}}\mr{F}(A) \to \mr{N}_{t_{1},t_{2}}\mr{F}(A)$ (which makes sense by \Cref{20170930-23}). Each column of \labelcref{20170930-13-eqn-01} is exact. By \Cref{20160705-35} below, we have that $I$ contains a nonzerodivisor of $A$; hence $I$ is not contained in any minimal prime of $A$ by \cite[Lemma (14.10)]{ALTMANKLEIMAN-ATOCA}; hence $A/I$ has smaller Krull dimension than that of $A$ (c.f. \cite[p. 15]{WEIBEL-KBOOK}); the image of $\alpha$ under $\mr{N}_{t_{1}}\mr{Pic}(A) \to \mr{N}_{t_{1}}\mr{Pic}(A/I)$ is contained in $\ker \mb{d}_{\mr{Pic}\mb{P}_{A/I}}^{1}$, which is $0$ by the induction hypothesis since $\dim A/I < \dim A$. Hence $\Delta_{t_{1}}(\alpha) = 0$, so by exactness of the left column of \labelcref{20170930-13-eqn-01}, there exists \[ \xi \in \mr{N}_{t_{1}}\mr{U}(A_{\lambda}/I) \] such that $\alpha = \partial_{t_{1}}(\xi)$. By e.g. \cite[I, Lemma 3.12]{WEIBEL-KBOOK} we have that $\xi$ is of the form \[ \xi = 1 + \beta(t_{1}) \] where \[ \beta \in t_{1}(\mr{nil}(A_{\lambda}/I)[t_{1}]) \] is a polynomial with nilpotent coefficients and whose constant coefficient is zero. We have \begin{align} \label{20170930-13-eqn-02} \mb{d}_{\mr{U}\mb{P}_{A_{\lambda}/I}}^{1}(\xi) = (1+\beta(t_{1}))(1-\beta(t_{1}t_{2}))(1+\beta(t_{2})) \end{align} in $\mr{N}_{t_{1},t_{2}}\mr{U}(A_{\lambda}/I)$. Since $\partial_{t_{1},t_{2}}(\mb{d}_{\mr{U}\mb{P}_{A_{\lambda}/I}}^{1}(\xi)) = \mb{d}_{\mr{Pic}\mb{P}_{A}}^{1}(\partial_{t_{1}}(\xi)) = \mb{d}_{\mr{Pic}\mb{P}_{A}}^{1}(\alpha) = 0$, by the exactness of the right column of \labelcref{20170930-13-eqn-01} there exists \[ \gamma \in \mr{N}_{t_{1},t_{2}}\mr{U}(A/I) \oplus \mr{N}_{t_{1},t_{2}}\mr{U}(A_{\lambda}) \] such that $\mb{d}_{\mr{U}\mb{P}_{A_{\lambda}/I}}^{1}(\xi) = \pm_{t_{1},t_{2}}(\gamma)$. Here by \cite[I, Lemma 3.12]{WEIBEL-KBOOK} the inclusion $\mr{U}(A_{\lambda}) \subset \mr{U}(A_{\lambda}[t_{1},t_{2}])$ is an equality since $A_{\lambda}$ is reduced, hence $\mr{N}_{t_{1},t_{2}}\mr{U}(A_{\lambda}) = 0$. Moreover $A/I \to A_{\lambda}/I$ is injective (since $I$ is the largest ideal of $A_{\lambda}$ contained in $A$), hence \labelcref{20170930-13-eqn-02} is in fact contained in $\mr{N}_{t_{1},t_{2}}\mr{U}(A/I)$. Thus in fact \[ \beta \in (A/I)[t_{1}] \] as can be seen for example by setting $t_{2} = 0$ in \labelcref{20170930-13-eqn-02}. In other words, we have that $\xi$ is in the image of $\pm_{t_{1}}$; since $\partial_{t_{1}} \circ \pm_{t_{1}} = 0$, we conclude $\alpha = 0$. \end{proof}

In the following lemma, we write out the details of a claim in \cite[p. 15]{WEIBEL-KBOOK}.

\begin{lemma} \label{20160705-35} Let $A$ be a ring with total ring of fractions $Q(A)$, and let $A \subset B \subset Q(A)$ be a subring. \begin{enumerate} \item[(i)] The inclusion $A \subset B$ preserves nonzerodivisors, and any nonzerodivisor of $B$ is of the form $r/u$ where $r,u \in A$ are nonzerodivisors of $A$. \item[(ii)] The total ring of fractions of $B$ is $Q(A)$. \item[(iii)] If $A \subset B$ is a finite extension, the conductor ideal $I = \{x \in A \;:\; xB \subset A\} = \Ann_{A}(B/A)$ contains a nonzerodivisor of $A$. \end{enumerate} \end{lemma} \begin{proof} (i) If $x \in A$ is a nonzerodivisor of $A$, then its image in $Q(A)$ is a nonzerodivisor of $Q(A)$, hence its image in $B$ is a nonzerodivisor of $B$. An arbitrary element of $B$ is of the form $r/u$ where $r,u \in A$ and $u$ is a nonzerodivisor of $A$. If $x \in A$ is an element such that $rx = 0$ in $A$, then $u(r/u)x = 0$ in $B$ and $u$ is a nonzerodivisor of $B$ (by the first part) so $(r/u)x = 0$ in $B$; then $x = 0$ since $r/u$ is by assumption a nonzerodivisor of $B$; hence $r$ is a nonzerodivisor of $A$. \par (ii) Let $\varphi : B \to S$ be a ring homomorphism such that $\varphi$ sends nonzerodivisors of $B$ to units of $S$. By the first part, $\varphi$ sends nonzerodivisors of $A$ to units of $S$, hence there exists a ring map $\xi : Q(A) \to S$ such that $\xi|_{A} = \varphi|_{A}$. By definition of $\xi$, for any $a/u \in B$ we have $\xi(a/u) = \varphi(a) \cdot (\varphi(u))^{-1} = \varphi(a/u)$; hence $\xi|_{B} = \varphi$. \par (iii) Let $x_{1}/u_{1} , \dotsc , x_{n}/u_{n}$ be elements of $B$ which generates $B$ as an $A$-module; then $u_{1} \dotsb u_{n}$ is a nonzerodivisor of $A$ which is contained in $I$. \end{proof}

\begin{remark} \label{20170930-36} In the proof of \Cref{20170930-13}, we may use the normalization instead of the seminormalization. If $A$ is a reduced finite type $\Z$-algebra, then its normalization $\overline{A} \subset Q(A)$ is a finite extension of $A$; thus $\overline{A}$ is a Noetherian reduced ring which is integrally closed in its total ring of fractions (by e.g. \Cref{20160705-35} (ii)), hence it is finite product of Noetherian normal integral domains \cite[030C]{STACKS-PROJECT}. It is easily checked from the definition of a seminormal ring that normal domains are seminormal. \end{remark}

\begin{remark} \label{20170930-26} For any ring $R$, by \cite[Lemma 1.5.1]{WEIBEL-PIACF} and \cite[Theorem 5.5]{WEIBEL-PIACF} we have an exact sequence \begin{align} \label{20170930-04-eqn-01} \begin{aligned} 0 &\to \Pic(R) \stackrel{f}{\to} \Pic(R[t]) \oplus \Pic(R[t^{-1}]) \stackrel{\Sigma}{\to} \Pic(R[t^{\pm}]) \\ &\to \H^{1}_{\et}(\Spec R,\Z) \to 0 \end{aligned} \end{align} of abelian groups, where $f$ denotes the map sending $\alpha \mapsto (\alpha,-\alpha)$ and $\Sigma$ denotes the addition map. For any ring $R$, by \cite[Theorem 2.4]{WEIBEL-PIACF} and \cite[Theorem 5.5]{WEIBEL-PIACF} we have isomorphisms \begin{align} \label{20170930-04-eqn-03} \H^{1}_{\et}(\Spec R,\Z) \simeq \H^{1}_{\et}(\Spec R[t],\Z) \simeq \H^{1}_{\et}(\Spec R[t^{\pm}],\Z) \end{align} of abelian groups. \end{remark}

The following is stated in \cite{WEIBEL-PIACF}; we write out the details here.

\begin{lemma} \label{20170930-15} Let $A$ be a strictly henselian local ring. Then the canonical map \begin{align*} \bigoplus_{(\ell,\Diamond) \in \{1,2\} \times \{+,-\}} \Pic(A[t_{\ell}^{\Diamond}]) \oplus \bigoplus_{(\Diamond_{1},\Diamond_{2}) \in \{+,-\}^{2}} \mr{N}_{t_{1}^{\Diamond_{1}},t_{2}^{\Diamond_{2}}}\mr{Pic}(A) \to \Pic(A[t_{1}^{\pm},t_{2}^{\pm}]) \end{align*} induced by the inclusions $A[t_{\ell}^{\Diamond}] \to A[t_{1}^{\pm},t_{2}^{\pm}]$ and $A[t_{1}^{\Diamond_{1}},t_{2}^{\Diamond_{2}}] \to A[t_{1}^{\pm},t_{2}^{\pm}]$ is an isomorphism. \end{lemma} \begin{proof} For notational convenience, we denote $t^{+} = t$ and $t^{-} = t^{-1}$, etc. Since $A$ is strictly henselian local, by \labelcref{20170930-04-eqn-03} the exact sequence \labelcref{20170930-04-eqn-01} reduces to an isomorphism \begin{align} \label{20170930-15-eqn-04} \Pic(A[t^{+}]) \oplus \Pic(A[t^{-}]) \stackrel{\sim}{\to} \Pic(A[t^{\pm}]) \end{align} and split exact sequences \begin{align} \label{20170930-15-eqn-05} 0 \to \Pic(A[t_{2}^{\Diamond}]) \to \Pic(A[t_{1}^{+},t_{2}^{\Diamond}]) \oplus \Pic(A[t_{1}^{-},t_{2}^{\Diamond}]) \to \Pic(A[t_{1}^{\pm},t_{2}^{\Diamond}]) \to 0 \end{align} and \begin{align} \label{20170930-15-eqn-06} 0 \to \Pic(A[t_{1}^{\pm}]) \to \Pic(A[t_{1}^{\pm},t_{2}^{+}]) \oplus \Pic(A[t_{1}^{\pm},t_{2}^{-}]) \to \Pic(A[t_{1}^{\pm},t_{2}^{\pm}]) \to 0 \end{align} by taking $R := A,A[t_{2}^{\Diamond}],A[t_{1}^{\pm}]$ respectively for $\Diamond \in \{+,-\}$. The sequence \labelcref{20170930-15-eqn-05} induces a natural isomorphism \begin{align} \label{20170930-15-eqn-07} \Pic(A[t_{2}^{\Diamond}]) \oplus \mr{N}_{t_{1}^{+}}\mr{Pic}(A[t_{2}^{\Diamond}]) \oplus \mr{N}_{t_{1}^{-}}\mr{Pic}(A[t_{2}^{\Diamond}]) \stackrel{\sim}{\to} \Pic(A[t_{1}^{\pm},t_{2}^{\Diamond}]) \end{align} of abelian groups. The isomorphism \labelcref{20170930-15-eqn-07} restricts to an isomorphism \begin{align} \label{20170930-15-eqn-09} \Pic(A[t_{2}^{\Diamond}]) \oplus \mr{N}_{t_{1}^{+},t_{2}^{\Diamond}}\mr{Pic}(A) \oplus \mr{N}_{t_{1}^{-},t_{2}^{\Diamond}}\mr{Pic}(A) \stackrel{\sim}{\to} \mr{N}_{t_{2}^{\Diamond}}\Pic(A[t_{1}^{\pm}]) \end{align} by taking the subgroups of elements annihilated by setting $t_{2}^{\Diamond} = 1$. The sequence \labelcref{20170930-15-eqn-06} induces a natural isomorphism \begin{align} \label{20170930-15-eqn-08} \Pic(A[t_{1}^{\pm}]) \oplus \mr{N}_{t_{2}^{+}}\mr{Pic}(A[t_{1}^{\pm}]) \oplus \mr{N}_{t_{2}^{-}}\mr{Pic}(A[t_{1}^{\pm}]) \stackrel{\sim}{\to} \Pic(A[t_{1}^{\pm},t_{2}^{\pm}]) \end{align} of abelian groups. We combine \labelcref{20170930-15-eqn-08} and \labelcref{20170930-15-eqn-04} and \labelcref{20170930-15-eqn-09} (for $\Diamond \in \{+,-\}$) and \Cref{20170930-35} to obtain the desired result. \end{proof}

\begin{lemma} \label{20170930-04} We have \[ \ker(\mb{d}_{\mr{Pic}\mb{L}_{A}}^{1}) = 0 \] for any strictly henselian local ring $A$. \end{lemma} \begin{proof} The inclusion $\mr{N}_{t_{1}}\mr{Pic}(A) \subseteq \mr{Pic}(A[t_{1}])$ is an equality since $A$ is a local ring; recall that $\mb{d}_{\mr{Pic}\mb{P}_{A}}^{1}(\mr{N}_{t_{1}}\mr{Pic}(A)) \subseteq \mr{N}_{t_{1},t_{2}}\mr{Pic}(A)$ by \Cref{20170930-23}. We have a commutative diagram \begin{center}\begin{tikzpicture}[>=angle 90] 
\matrix[matrix of math nodes,row sep=4em, column sep=3em, text height=1.5ex, text depth=0.25ex] { 
|[name=11]| (\Pic(A[t_{1}]))^{\oplus 2} & |[name=12]| \Pic(A[t_{1}^{\pm}]) \\ 
|[name=21]| (\mr{N}_{t_{1},t_{2}}\mr{Pic}(A))^{\oplus 2} & |[name=22]| \Pic(A[t_{1}^{\pm},t_{2}^{\pm}]) \\
}; 
\draw[->,font=\scriptsize]
(11) edge node[above=-1pt] {$\simeq$} node[below=-1pt] {$f_{1}$} (12) (21) edge[right hook->] node[below=-1pt] {$f_{2}$} (22) (11) edge node[left=-1pt] {$(\mb{d}_{\mr{Pic}\mb{P}_{A}}^{1})^{\oplus 2}$} (21) (12) edge node[right=-1pt] {$\mb{d}_{\mr{Pic}\mb{L}_{A}}^{1}$} (22); \end{tikzpicture} \end{center} where $f_{1}$ and $f_{2}$ are the addition maps induced on the Picard groups by the $A$-algebra maps $A[t_{1}] \to A[t_{1}^{\pm}]$ sending $t_{1}$ to $t_{1},t_{1}^{-1}$ and $A[t_{1},t_{2}] \to A[t_{1}^{\pm},t_{2}^{\pm}]$ sending $(t_{1},t_{2}) \mapsto (t_{1},t_{2}),(t_{1}^{-1},t_{2}^{-1})$ respectively. Here $f_{1}$ is an isomorphism by \labelcref{20170930-04-eqn-01} since $A$ is strictly henselian local, and $f_{2}$ is injective by \Cref{20170930-15}. Since $\mb{d}_{\mr{Pic}\mb{P}_{A}}^{1}$ is injective by \Cref{20170930-13}, we have that $\mb{d}_{\mr{Pic}\mb{L}_{A}}^{1}$ is injective. \end{proof}

\section{Unit groups of Laurent polynomial rings} \label{sec-04}

The purpose of this section is to prove \Cref{20170930-16}. As in \Cref{sec-03}, when it is convenient we will denote $\mr{U}(A) := A^{\times}$ the group of units of a ring $A$.

\begin{lemma} \label{20170930-20} Let $\{A_{\lambda}\}_{\lambda \in \Lambda}$ be a filtered inductive system of rings, and let \[ \textstyle A := \varinjlim_{\lambda \in \Lambda} A_{\lambda} \] be the colimit ring. In the notation of \Cref{20170930-33} and \Cref{20170930-14}, the induced morphism of complexes \[ \textstyle \varinjlim_{\lambda \in \Lambda} \mr{C}^{\bullet}(\mr{U}\mb{L}_{A_{\lambda}}) \to \mr{C}^{\bullet}(\mr{U}\mb{L}_{A}) \] is an isomorphism. \end{lemma} \begin{proof} For any $n \ge 0$, the functor $(\mr{Ring}) \to (\mr{Ab})$ sending $A \mapsto (A[t_{1}^{\pm},\dotsc,t_{n}^{\pm})^{\times}$ is locally of finite presentation. \end{proof}

\begin{lemma} \label{20170930-16} For any ring $A$, we have $\mathsf{h}^{2}(\mr{C}^{\bullet}(\mr{U}\mb{L}_{A})) = 0$. \end{lemma} \begin{proof} By writing $A$ as the filtered colimit of subrings which are finite type $\Z$-algebras, by \Cref{20170930-20} we may reduce to the case when $A$ is a finite type $\Z$-algebra. By replacing $\Spec A$ by a connected component, we may assume that $\Spec A$ is connected. Let $\mf{n} \subset A$ be the nilradical of $A$. By \cite[Corollary 6]{NEHER-IANEITGAOAUPG}, a unit $\xi$ of $A[t_{1}^{\pm},t_{2}^{\pm}]$ is of the form \begin{align} \label{20170930-16-eqn-02} \xi = ut_{1}^{e_{1}}t_{2}^{e_{2}} + x(t_{1},t_{2}) \end{align} where $u \in A^{\times}$ is a unit and $(e_{1},e_{2}) \in \Z^{\oplus 2}$ is an ordered pair of integers and $x(t_{1},t_{2}) \in \mf{n}A[t_{1}^{\pm},t_{2}^{\pm}]$ is a Laurent polynomial all of whose coefficients are nilpotent. We have that each unit $u \in A^{\times} \subset (A[t_{1}^{\pm},t_{2}^{\pm}])^{\times}$ is in the image of $\mb{d}_{\mr{U}\mb{L}_{A}}^{1}$, namely the image of the unit $u \in A^{\times} \subset (A[t_{1}^{\pm}])^{\times}$ since $u \cdot u^{-1} \cdot u = u$. Hence we may assume that the unit $u$ of \labelcref{20170930-16-eqn-02} is equal to $1$. Suppose $\xi \in \ker \mb{d}_{\mr{U}\mb{L}_{A/\mf{n}}}^{2}$. By reduction $A \to A/\mf{n}$ we have that \begin{align*} \mb{d}_{\mr{U}\mb{L}_{A}}^{2}(ut_{1}^{e_{1}}t_{2}^{e_{2}}) &= (ut_{1}^{e_{1}}t_{2}^{e_{2}}) \cdot (ut_{1}^{e_{1}}(t_{2}t_{3})^{e_{2}})^{-1} \cdot (u(t_{1}t_{2})^{e_{1}}t_{3}^{e_{2}}) \cdot (ut_{2}^{e_{1}}t_{3}^{e_{2}})^{-1} \\ &= t_{1}^{e_{1}}t_{3}^{-e_{2}} \end{align*} must be equal to $1$, hence $e_{1} = e_{2} = 0$. This implies that $\mathsf{h}^{2}(\mr{C}^{\bullet}(\mr{U}\mb{L}_{A/\mf{n}^{1}})) = 0$. \par We have a sequence \[ A/\mf{n}^{s} \to A/\mf{n}^{s-1} \to \dotsb \to A/\mf{n}^{2} \to A/\mf{n}^{1} \] where each map is a surjective ring map with square-zero kernel. Hence, since the complex $\mr{C}^{\bullet}(\mr{U}\mb{L}_{A})$ is functorial in $A$, it suffices to show that, for any ring $A$ and ideal $I \subset A$ satisfying $I^{2} = 0$, if $\mathsf{h}^{2}(\mr{C}^{\bullet}(\mr{U}\mb{L}_{A/I})) = 0$ then $\mathsf{h}^{2}(\mr{C}^{\bullet}(\mr{U}\mb{L}_{A})) = 0$. The quotient $A \to A/I$ induces a morphism $\mr{C}^{\bullet}(\mr{U}\mb{L}_{A}) \to \mr{C}^{\bullet}(\mr{U}\mb{L}_{A/I})$ of complexes of abelian groups, part of which is a commutative diagram \begin{equation} \label{20170930-16-eqn-03} \begin{tikzpicture}[>=angle 90, baseline=(current bounding box.center)] 
\matrix[matrix of math nodes,row sep=3.5em, column sep=2em, text height=1.7ex, text depth=0.5ex] { 
|[name=11]| (A[t_{1}^{\pm}])^{\times} & |[name=12]| (A[t_{1}^{\pm},t_{2}^{\pm}])^{\times} & |[name=13]| (A[t_{1}^{\pm},t_{2}^{\pm},t_{3}^{\pm}])^{\times} \\ 
|[name=21]| ((A/I)[t_{1}^{\pm}])^{\times} & |[name=22]| ((A/I)[t_{1}^{\pm},t_{2}^{\pm}])^{\times} & |[name=23]| ((A/I)[t_{1}^{\pm},t_{2}^{\pm},t_{3}^{\pm}])^{\times} \\ 
}; 
\draw[->,font=\scriptsize]
(11) edge node[above=0pt] {$\mb{d}_{\mr{U}\mb{L}_{A}}^{1}$} (12) (12) edge node[above=0pt] {$\mb{d}_{\mr{U}\mb{L}_{A}}^{2}$} (13) (21) edge node[above=0pt] {$\mb{d}_{\mr{U}\mb{L}_{A/I}}^{1}$} (22) (22) edge node[above=0pt] {$\mb{d}_{\mr{U}\mb{L}_{A/I}}^{2}$} (23) (11) edge[->>] node[left=-1pt] {$\pi^{1}$} (21) (12) edge[->>] node[left=-1pt] {$\pi^{2}$} (22) (13) edge[->>] node[left=-1pt] {$\pi^{3}$} (23); \end{tikzpicture} \end{equation} where each vertical arrow $\pi^{1},\pi^{2},\pi^{3}$ is surjective since $I$ is square-zero. By a diagram chase on \labelcref{20170930-16-eqn-03}, to show that the top row is exact it suffices to show that every element of $(\ker \mb{d}_{\mr{U}\mb{L}_{A}}^{2}) \cap (\ker \pi^{2})$ is in the image of $\mb{d}_{\mr{U}\mb{L}_{A}}^{1}$. We have $\ker \pi^{2} = 1 + IA[t_{1}^{\pm},t_{2}^{\pm}]$; moreover, since $I$ is square-zero, the (multiplicative) condition that $1 + x(t_{1},t_{2}) \in \ker \mb{d}_{\mr{U}\mb{L}_{A}}^{2}$ is equivalent to the (additive) condition that the element \begin{align} \label{20170930-16-eqn-04} x(t_{1},t_{2}) - x(t_{1},t_{2}t_{3}) + x(t_{1}t_{2},t_{3}) - x(t_{2},t_{3}) \end{align} of $A[t_{1}^{\pm},t_{2}^{\pm},t_{3}^{\pm}]$ is equal to zero. Let \begin{align*} \mr{H}_{0} &:= \{e_{3} = 0\} \\ \mr{H}_{1} &:= \{e_{2} = e_{3}\} \\ \mr{H}_{2} &:= \{e_{1} = e_{2}\} \\ \mr{H}_{3} &:= \{e_{1} = 0\} \end{align*} be hyperplanes of $\Z^{\oplus 3} = \{(e_{1},e_{2},e_{3})\}$ defined by the equations corresponding to the maps $\mb{L}_{A}(\mr{p}_{0}^{2}) , \mb{L}_{A}(\mr{p}_{1}^{2}) , \mb{L}_{A}(\mr{p}_{2}^{2}) , \mb{L}_{A}(\mr{p}_{3}^{2})$ in the sense that the image of $\Z^{\oplus 2}$ under $\mb{L}_{A}(\mr{p}_{i}^{2})$ is $\mr{H}_{i} \subset \Z^{\oplus 3}$. Then the pairwise intersections \begin{align*} \label{20170930-16-eqn-05} \begin{aligned} \mr{H}_{0} \cap \mr{H}_{1} &= \Z(1,0,0) &&& \mr{H}_{1} \cap \mr{H}_{2} &= \Z(1,1,1) \\ \mr{H}_{0} \cap \mr{H}_{2} &= \Z(1,1,0) &&& \mr{H}_{1} \cap \mr{H}_{3} &= \Z(0,1,1) \\ \mr{H}_{0} \cap \mr{H}_{3} &= \Z(0,1,0) &&& \mr{H}_{2} \cap \mr{H}_{3} &= \Z(0,0,1) \end{aligned} \end{align*} are all distinct. Let \[ x_{e_{1},e_{2}} \in I \] be the coefficient of $t_{1}^{e_{1}}t_{2}^{e_{2}}$ in $x(t_{1},t_{2})$. Then if $(e_{1},e_{2}) \in \Z^{\oplus 2}$ is an ordered pair for which $x_{e_{1},e_{2}} \ne 0$, then we must have \[ (e_{1},e_{2}) \in \Z(1,0) \cup \Z(1,1) \cup \Z(0,1) \] in $\Z^{\oplus 2}$. Moreover, saying that \labelcref{20170930-16-eqn-04} is equal to zero translates to the collection of equations \begin{align*} \label{20170930-16-eqn-06} \begin{aligned} x_{e,0} - x_{e,0} &= 0 &&& x_{e,e} - x_{e,e} &= 0 \\ x_{e,e} + x_{e,0} &= 0 &&& x_{0,e} + x_{e,e} &= 0 \\ x_{0,e} - x_{e,0} &= 0 &&& x_{0,e} - x_{0,e} &= 0 \end{aligned} \end{align*} for all $e \in \Z$, which simplifies to \[ x_{e,0} = x_{0,e} = -x_{e,e} \] for all $e \in \Z$. Then \[ \textstyle 1+x(t_{1},t_{2}) = \mb{d}_{\mr{U}\mb{L}_{A}}^{1}(1-\sum_{e \in \Z} x_{e,e}t_{1}^{e}) \] so we have the desired result. \end{proof}

\section{Proof of the main theorem} \label{sec-05}

In this section we prove \Cref{20170930-01}.

Our argument, in outline, is that of the proof of \cite[II, Lemma 1\ensuremath{'}]{GABBER-THESIS}. Namely, we compute $\H^{2}_{\et}(\mc{G} , \G_{m,\mc{G}})$ using the Leray spectral sequence associated to the map $\pi$ and sheaf $\G_{m,\mc{G}}$, which is of the form \begin{align} \label{sec-05-eqn-01} \mr{E}_{2}^{p,q} = \H^{p}_{\et}(S , \mb{R}^{q}\pi_{\ast}\G_{m,\mc{G}}) \implies \H^{p+q}_{\et}(\mc{G} , \G_{m,\mc{G}}) \end{align} with differentials $\mr{d}_{2}^{p,q} : \mr{E}_{2}^{p,q} \to \mr{E}_{2}^{p+2,q-1}$. 

The stalks of $\mb{R}^{2}\pi_{\ast}\G_{m,\mc{G}}$ are described by \Cref{20170930-08}.

\begin{setup}[Descent spectral sequence for $\mr{B}\G_{m}$] \label{20170930-27} Let $A$ be a ring and let \[ \xi : \Spec A \to \mr{B}\G_{m,A} \] be the smooth cover associated to the trivial $\G_{m,A}$-torsor. The cohomological descent spectral sequence associated to $\xi$ gives a spectral sequence \begin{align} \label{20170930-27-eqn-01} \mr{E}_{1}^{p,q} = \H^{q}_{\et}(\G_{m,A}^{\times p} , \G_{m}) \implies \H^{p+q}_{\et}(\mr{B}\G_{m,A},\G_{m}) \end{align} where the $q$th row $\mr{E}_{1}^{\bullet,q} = \H^{q}_{\et}(\G_{m,A}^{\times \bullet} , \G_{m})$ can be realized as the complex $\mr{C}^{\bullet}(\mr{F}\mb{L}_{A})$ (see \Cref{20170930-33} and \Cref{20170930-14}) where the functor $\mr{F} : (\mr{Ring}) \to (\mr{Ab})$ is defined by $\mr{F}(R) := \H^{q}_{\et}(\Spec R , \G_{m})$. The lower-left part of the $\mr{E}_{1}$-page of the spectral sequence \labelcref{20170930-27-eqn-01} is \begin{center}\begin{tikzpicture}[>=angle 90] 
\matrix[matrix of math nodes,row sep=1em, column sep=1.3em, text height=1.5ex, text depth=0.25ex] { 
|[name=11]| \H^{3}_{\et}(\G_{m,A}^{\times 0},\G_{m}) & |[name=12]| \H^{3}_{\et}(\G_{m,A}^{\times 1},\G_{m}) & |[name=13]| \H^{3}_{\et}(\G_{m,A}^{\times 2},\G_{m}) & |[name=14]| \H^{3}_{\et}(\G_{m,A}^{\times 3},\G_{m}) \\ 
|[name=21]| \H^{2}_{\et}(\G_{m,A}^{\times 0},\G_{m}) & |[name=22]| \H^{2}_{\et}(\G_{m,A}^{\times 1},\G_{m}) & |[name=23]| \H^{2}_{\et}(\G_{m,A}^{\times 2},\G_{m}) & |[name=24]| \H^{2}_{\et}(\G_{m,A}^{\times 3},\G_{m}) \\ 
|[name=31]| \H^{1}_{\et}(\G_{m,A}^{\times 0},\G_{m}) & |[name=32]| \H^{1}_{\et}(\G_{m,A}^{\times 1},\G_{m}) & |[name=33]| \H^{1}_{\et}(\G_{m,A}^{\times 2},\G_{m}) & |[name=34]| \H^{1}_{\et}(\G_{m,A}^{\times 3},\G_{m}) \\ 
|[name=41]| \H^{0}_{\et}(\G_{m,A}^{\times 0},\G_{m}) & |[name=42]| \H^{0}_{\et}(\G_{m,A}^{\times 1},\G_{m}) & |[name=43]| \H^{0}_{\et}(\G_{m,A}^{\times 2},\G_{m}) & |[name=44]| \H^{0}_{\et}(\G_{m,A}^{\times 3},\G_{m}) \\ 
}; 
\draw[->,font=\scriptsize]
(11) edge (12) (12) edge (13) (13) edge (14)
(21) edge (22) (22) edge (23) (23) edge (24)
(31) edge (32) (32) edge (33) (33) edge (34)
(41) edge (42) (42) edge (43) (43) edge (44); \end{tikzpicture} \end{center} where $\mr{d}_{1}^{0,q}$ is the zero map for all $q \ge 0$ since $\mr{B}\G_{m,A}$ is the quotient of $\Spec A$ by the trivial action of $\G_{m,A}$. \end{setup}

\begin{lemma} \label{20170930-08} Assume the setup of \Cref{20170930-27}. For any strictly henselian local ring $A$, we have $\H^{2}_{\et}(\mr{B}\G_{m,A},\G_{m}) = 0$. \end{lemma} \begin{proof} We have $\mr{E}_{1}^{0,q} = \H^{q}_{\et}(\Spec A , \G_{m}) = 0$ for any $q \ge 1$ since $A$ is strictly henselian. We have $\mr{E}_{2}^{1,1} = 0$ by \Cref{20170930-04} and $\mr{E}_{2}^{2,0} = 0$ by \Cref{20170930-16}. \end{proof}

\begin{remark} \label{20170930-25} We show that, in the proof of \Cref{20170930-08}, it is possible to reduce to the case when $A$ is a reduced ring; the reducedness assumption simplifies the proof of \Cref{20170930-16}. By standard limit arguments, we may assume that $A$ is a finite type $\Z$-algebra. Then the reduction $A \to A_{\mr{red}}$ can be factored as a finite sequence of square-zero thickenings. Thus we reduce to showing that if $A \to A_{0}$ is a surjection of rings whose kernel $I$ is square-zero, then the reduction map \begin{align} \label{20170930-25-eqn-02} &\H^{2}_{\et}(\mr{B}\G_{m,A} , \G_{m}) \to \H^{2}_{\et}(\mr{B}\G_{m,A_{0}} , \G_{m}) \end{align} is an isomorphism. Set $\ms{X} := \mr{B}\G_{m,A}$ and $\ms{X}_{0} := \mr{B}\G_{m,A_{0}}$ and let $i : \ms{X}_{0} \to \ms{X}$ be the closed immersion. We may use either the big \'etale site $(\mr{Sch}/\ms{X})_{\et}$ or the lisse-\'etale site $\LisEt(\ms{X})$ to compute cohomology on $\ms{X}$, since the inclusion functor of sites \[ u : \LisEt(\ms{X}) \to (\mr{Sch}/\ms{X})_{\et} \] induces a restriction functor on abelian sheaves \[ u^{-1} : \mr{Ab}((\mr{Sch}/\ms{X})_{\et}) \to \mr{Ab}(\LisEt(\ms{X})) \] which is exact and admits an exact left adjoint $u_{!}$ (see \cite[0788 (1)]{STACKS-PROJECT}). There is an exact sequence \[1 \to 1+I \to \G_{m,\ms{X}} \to i_{\ast}\G_{m,\ms{X}_{0}} \to 1 \] of abelian sheaves on $\LisEt(\ms{X})$; here left exactness follows from the fact that for any scheme $X$ and smooth morphism $X \to \ms{X}$ the composition $X \to \ms{X} \to \Spec A$ is flat. We have an induced long exact sequence \[ \dots \to \H^{p}_{\et}(\ms{X},I) \to \H^{p}_{\et}(\ms{X},\G_{m,\ms{X}}) \to \H^{p}_{\et}(\ms{X},i_{\ast}\G_{m,\ms{X}_{0}}) \to \H^{p+1}_{\et}(\ms{X},I) \to \dotsb \] in cohomology. We have $\H^{p}_{\et}(\ms{X},i_{\ast}\G_{m,\ms{X}_{0}}) \simeq \H^{p}_{\et}(\ms{X}_{0},\G_{m,\ms{X}_{0}})$ for $p \ge 0$ since pushforward along a closed immersion in the \'etale topology is exact (using e.g. \cite[04E3]{STACKS-PROJECT}). It suffices now to show that if $\ms{F}$ is any quasi-coherent $\mc{O}_{\ms{X}}$-module then $\H^{p}_{\et}(\ms{X},\ms{F}) = 0$ for all $p > 0$. The category of quasi-coherent $\mc{O}_{\ms{X}}$-modules corresponds to the category $\mc{C}$ of $\Z$-graded $A$-modules. Denoting by $\pi : \ms{X} \to \Spec A$ the structure map, the pushforward functor $\pi_{\ast} : \mr{QCoh}(\ms{X}) \to \mr{QCoh}(A)$ corresponds to sending a $\Z$-graded module $M_{\bullet} = \bigoplus_{n \in \Z} M_{n}$ to the degree zero component $M_{0}$. Since this is an exact functor, we have that $\pi$ is cohomologically affine \cite[Definition 3.1]{ALPER-GMSFAS}. Since $\pi$ has affine diagonal, we have the desired result by \cite[Remark 3.5]{ALPER-GMSFAS}. \end{remark}

\begin{pg}[Proof of {\Cref{20170930-01}}] \label{01} For any strictly henselian local ring $A$, we have \[ \H^{2}_{\et}(\mr{B}\G_{m,A} , \G_{m,\mr{B}\G_{m,A}}) = 0 \] by \Cref{20170930-08}, hence \[ \mb{R}^{2}\pi_{\mc{G},\ast}\G_{m,\mc{G}} = 0 \] since its stalks vanish. By \Cref{20171214-02}, we have \[ \mb{R}^{1}\pi_{\mc{G},\ast}\G_{m,\mc{G}} \simeq \underline{\Hom}_{\mr{Ab}(S)}(\G_{m,S},\G_{m,S}) = \underline{\Z} \] so the Leray spectral sequence \labelcref{sec-05-eqn-01} gives an exact sequence \begin{align} \label{01-eqn-01} \H^{0}_{\et}(S , \underline{\Z}) \stackrel{\dagger}{\to} \H^{2}_{\et}(S , \G_{m,S}) \stackrel{\pi_{\mc{G}}^{\ast}}{\to} \H^{2}_{\et}(\mc{G} , \G_{m,\mc{G}}) \to \H^{1}_{\et}(S , \underline{\Z}) \end{align} where by \Cref{20171214-03} the first map $\dagger$ sends $1 \mapsto [\mc{G}]$. By \Cref{20171216-03} the last term $\H^{1}_{\et}(S , \underline{\Z})$ is a torsion-free abelian group. If $\alpha \in \H^{2}_{\et}(S , \G_{m,S})$ is a class such that $\pi_{\mc{G}}^{\ast}(\alpha)$ is $n$-torsion for some $n \in \Z_{\ge 0}$, then $\pi_{\mc{G}}^{\ast}(n\alpha) = 0$, hence $n\alpha = m[\mc{G}]$ for some $m \in \Z_{\ge 0}$; since by assumption $[\mc{G}]$ is a torsion class, we have that $\alpha$ is torsion; in other words the restriction $\pi_{\mc{G}}^{\ast} : \H^{2}_{\et}(S , \G_{m,S})_{\mr{tors}} \to \H^{2}_{\et}(\mc{G} , \G_{m,\mc{G}})_{\mr{tors}}$ is surjective. Hence we have the desired result. \qed \end{pg}

\begin{remark} \label{20180331-02} As pointed out to me by Siddharth Mathur, in \Cref{20170930-01}, the restriction map \[ \pi_{\mc{G}}^{\ast} : \Br'(S) \to \Br'(\mc{G}) \] is not necessarily surjective if $[\mc{G}] \in \H_{\et}^{2}(S,\G_{m,S})$ is a nontorsion class. Let $S$ be a scheme for which $\H_{\et}^{2}(S,\G_{m,S})$ is not a torsion group; let $\alpha \in \H_{\et}^{2}(S,\G_{m,S})$ be a nontorsion element, and let $\pi_{\mc{G}} : \mc{G} \to S$ be the $\G_{m,S}$-gerbe corresponding to the class $2\alpha \in \H_{\et}^{2}(S,\G_{m,S})$. Then $\pi_{\mc{G}}^{\ast}(\alpha)$ is a 2-torsion class of $\H_{\et}^{2}(\mc{G},\G_{m,\mc{G}})$. We show that there does not exist any torsion element $\beta \in \H_{\et}^{2}(S,\G_{m,S})$ such that $\pi_{\mc{G}}^{\ast}(\alpha) = \pi_{\mc{G}}^{\ast}(\beta)$. If so, then $\alpha-\beta = n[\mc{G}] = 2n\alpha$ for some $n$, which means $(2n-1)\alpha$ is torsion, which contradicts our assumption that $\alpha$ is nontorsion. Taking as our $S$ above the normal surface of Mumford \cite[Remarques 1.11, b]{Gro68b} for which $\H_{\et}^{2}(S,\G_{m,S})$ is not a torsion group, we obtain an example of a $\G_{m,S}$-gerbe $\pi_{\mc{G}} : \mc{G} \to S$ for which the restriction \[ \pi_{\mc{G}}^{\ast} : \H_{\et}^{2}(S,\G_{m,S}) \to \H_{\et}^{2}(\mc{G},\G_{m,\mc{G}}) \] is surjective (by \labelcref{01-eqn-01}, using that $\H_{\et}^{1}(S,\underline{\Z}) = 0$ by \cite[VIII, Prop. 5.1]{SGA7-I} since $S$ is geometrically unibranch) but the restriction to the torsion subgroups is not surjective. \qed \end{remark}

\appendix

\section{Torsors under torsion-free abelian groups} \label{sec-06}

In \cite[Corollary 7.9.1]{WEIBEL-PIACF}, it is proved that $\H^{1}_{\et}(S,\underline{\Z})$ is a torsion-free abelian group if $S$ is a quasi-compact quasi-separated scheme. In this section, we record a different proof which works over an arbitrary site. This argument is from \cite[093J]{STACKS-PROJECT}.

Let $\mc{S}$ be a site. For any set $\mr{S}$, let $\underline{\mr{S}}$ denote the constant sheaf on $\mc{S}$ associated to $\mr{S}$. 

\begin{lemma} \label{20171216-01} Let $f : \mr{S} \to \mr{T}$ be a surjective function between sets. Then the induced map \[ \Gamma(\mc{S} , f) : \Gamma(\mc{S},\underline{\mr{S}}) \to \Gamma(\mc{S},\underline{\mr{T}}) \] is surjective. \end{lemma} \begin{proof} Choose a function $g : \mr{T} \to \mr{S}$ satisfying $fg = \id_{\mr{T}}$. By functoriality of the ``constant sheaf'' functor, we have $\Gamma(\mc{S},f) \circ \Gamma(\mc{S},g) = \id_{\Gamma(\mc{S},\underline{\mr{T}})}$. \end{proof}

\begin{lemma} \label{20171216-02} Let $0 \to \mr{A} \to \mr{B} \to \mr{C} \to 0$ be an exact sequence of abelian groups. Then the induced map \[ \H^{1}(\mc{S} , \underline{\mr{A}}) \to \H^{1}(\mc{S} , \underline{\mr{B}}) \] is injective. \end{lemma} \begin{proof} As part of the long exact sequence in cohomology, we obtain an exact sequence \[ \Gamma(\mc{S},\underline{\mr{B}}) \to \Gamma(\mc{S},\underline{\mr{C}}) \to \H^{1}(\mc{S} , \underline{\mr{A}}) \to \H^{1}(\mc{S} , \underline{\mr{B}}) \] where the first arrow is surjective by \Cref{20171216-01}, hence the third arrow is injective. \end{proof}

\begin{lemma} \label{20171216-03} Let $\mr{A}$ be a torsion-free abelian group. Then $\H^{1}(\mc{S},\underline{\mr{A}})$ is a torsion-free abelian group. \end{lemma} \begin{proof} Let $n$ be a positive integer. Applying \Cref{20171216-02} to the exact sequence \[ 0 \to \mr{A} \stackrel{\times n}{\to} \mr{A} \to \mr{A}/n\mr{A} \to 0 \] implies that the multiplication-by-$n$ map on $\H^{1}(\mc{S},\underline{\mr{A}})$ is injective. \end{proof}

\bibliography{../allbib}
\bibliographystyle{alpha}

\end{document}